\documentclass[journal/twoside,web]{ieeecolor}
\usepackage{lcsys}
\usepackage{cite}
\usepackage{amsmath,amssymb,amsfonts}
\usepackage{graphicx}
\usepackage{textcomp}
\usepackage{amssymb}
\usepackage{amsmath}
\usepackage{mdframed}

\usepackage{graphicx}
\newcommand{\bbsym}[1]{\ensuremath{\boldsymbol{#1}}}
\usepackage{xcolor}
\usepackage{makecell}
\usepackage{comment}
\usepackage{todonotes}

\usepackage{cite}
\usepackage{graphicx}
\usepackage{subcaption}

\usepackage{algorithm}
\usepackage{algorithmic}

\usepackage{xcolor}
\definecolor{light-blue}{rgb}{0.3,0.5,0.8}
\usepackage[colorlinks=true, linkcolor=cyan, citecolor=cyan, filecolor=magenta, urlcolor=cyan]{hyperref}

\def\BibTeX{{\rm B\kern-.05em{\sc i\kern-.025em b}\kern-.08em
    T\kern-.1667em\lower.7ex\hbox{E}\kern-.125emX}}
\title{A Feasibility Analysis at Signal-Free Intersections}
\author{Filippos N. Tzortzoglou$^1$, \textit{Student Member, IEEE,} Logan E. Beaver$^2$, \textit{Member, IEEE,} and\\
Andreas A. Malikopoulos$^1$, \textit{Senior Member, IEEE} 
\thanks{This work was supported by NSF under Grants CNS-2149520 and CMMI-2219761.} 
\thanks{$^1$Filippos N. Tzortzoglou and Andreas A. Malikopoulos are with the School of Civil and Environmental Engineering, Cornell University, Ithaca, NY 14853 USA.(emails: \tt\small{ft254@cornell.edu; amaliko@cornell.edu;)}}
\thanks{$^1$Logan E. Beaver is with the Department of Mechanical and Aerospace Engineering, Old Dominion University, Norfolk, VA 23529  (emails:\tt\small{ lbeaver@odu.edu
 )}}}

\usepackage{amsthm}

\newtheorem{theorem}{Theorem}
\newtheorem{proposition}{Proposition}
 \newtheorem{problem}{Problem}
 \newtheorem{remark}{Remark}
 \newtheorem{lemma}{Lemma}

\begin{document}

  \maketitle
 \thispagestyle{empty}
 \pagestyle{empty}
 
\begin{abstract}
In this letter, we address the problem of improving the feasible domain of the solution of a decentralized control framework for coordinating connected and automated vehicles (CAVs) at signal-free intersections as the traffic volume increases. 
The framework provides the optimal trajectories of CAVs to cross the intersection safely without stop-and-go driving. However, as the traffic volume increases, the domain of the feasible trajectories decreases. 
We use concepts of numerical interpolation to identify appropriate polynomials that can serve as alternative trajectories of the CAVs, expanding the domain of the feasible CAV trajectories. We provide the conditions under which such polynomials exist.
Finally, we demonstrate the efficacy of our approach through numerical simulations.
\end{abstract}
\vspace{-5pt}

\section{ INTRODUCTION}

In recent decades, society has witnessed a commendable effort towards improving emerging mobility systems \cite{chow2018informed}. One significant sector of this improvement is associated with the use of connected and automated vehicles (CAVs) \cite{guanetti2018control}. It has been shown that the use of CAVs can lead to a safer traffic network, improve fuel consumption, and maximize throughput on the roads \cite{du2023adaptive,rios2018impact,sarantinoudis2023ros}. 
Several research efforts have provided approaches on how CAVs can create a more sustainable traffic network \cite{shiwakoti2020investigating}. Adaptive cruise control and cooperative adaptive cruise control have been extensively discussed in the literature \cite{xiao2010comprehensive}. These approaches allow for a vehicle to define its speed/orientation trajectories according to the environment in which it exists \cite{tzortzoglou2023performance}. Attention has been given to ecological adaptive cruise control \cite{dollar2018efficient}, which also considers environmental aspects of traveling behavior. 

Recent studies have shown that a crucial reason for bottlenecks in the US is associated with the congestion on merging on-ramps, intersections, and roundabouts \cite{schrank2015urban}. Many researchers have investigated how CAVs can contribute to such scenarios using different kinds of approaches like time-optimal control strategies \cite{xiao2021decentralized, Le2023Stochastic}, reinforcement learning \cite{katriniok2022towards,10241682} and model predictive control \cite{hult2018optimal}.

A decentralized optimal control framework that minimizes fuel consumption and maximizes the throughput of CAVs at signal-free intersections was reported in \cite{Malikopoulos2020} while the problem for adjacent intersections was addressed in \cite{chalaki2020TCST}. \textit{Sun et. al} \cite{sun2020optimal} and \textit{Meng et. al} \cite{meng2020eco} studied the efficiency of eco-driving control at signalized intersections, whereas \textit{Abduljabbar et. al} \cite{abduljabbar2021control} explored the operation of CAVs under roundabouts. \textit{Bang et. al} \cite{Bang2022combined} explored a combination of coordination and eco-routing at signal-free intersections.

In this letter, we analyze the decentralized control framework reported in \cite{Malikopoulos2020} and address its limitations related to operating feasibility in scenarios with high traffic volume. The framework provides the optimal time trajectory for each CAV that arrives at an intersection's control region. However, this framework encounters a significant constraint in instances where such a time trajectory is unattainable. To address this issue, we introduce an alternative approach that enables CAVs to determine their trajectories online when no feasible solution exists. Our approach extends the problem formulation in \cite{Malikopoulos2020}, which confines the optimal unconstrained trajectory to a $3$rd-order polynomial. By recognizing that a $3$rd order polynomial solution can be restrictive, employing a higher-order polynomial may facilitate the identification of an unconstrained trajectory. Using concepts from numerical mathematics, we show that any CAV entering the intersection can extend its range of feasible solutions by interpolating a polynomial that describes its position trajectory. Nonetheless, given the potential for multiple trajectory solutions, our approach prioritizes the trajectory that minimizes jerk, thus optimizing energy efficiency and passenger comfort. Moreover, we demonstrate the feasibility of solving this optimization problem online, enabling real-time trajectory planning for vehicles. The contributions of this letter are:
\begin{itemize}
    \item Proposing a control framework for managing CAVs at signal-free intersections, ensuring feasibility even under conditions of high traffic volume.
    \item Providing conditions for deriving an unconstrained trajectory for the CAVs through numerical interpolation.
    \item Establishing an optimization problem that facilitates the real-time determination of optimal CAV trajectories.
\end{itemize}

The remainder of the letter proceeds as follows. In Section II, we introduce the problem formulation. In Section III, we provide our proposed approach, and in Section IV, we provide numerical simulations. Finally, in Section V, we draw some concluding remarks.

\begin{figure}
\centering
\includegraphics[width=0.4\textwidth]{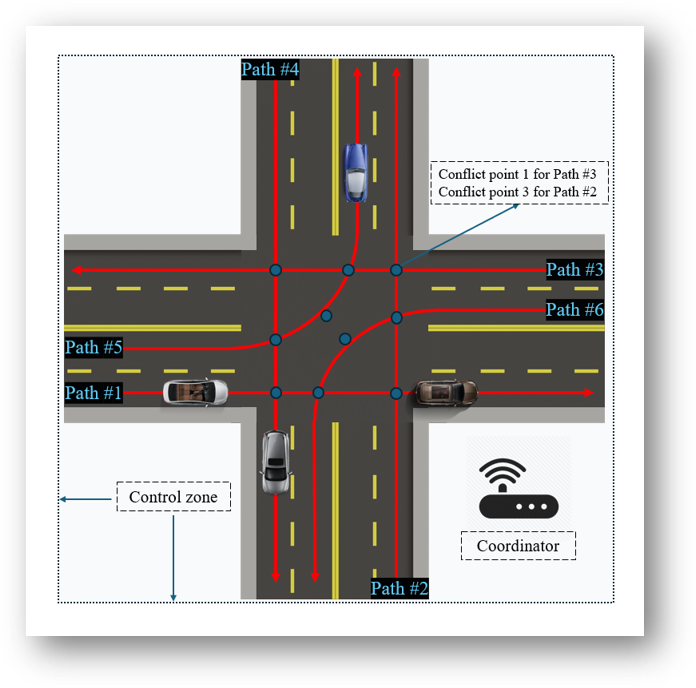}
\caption{Schematic of the intersection showing the control zone, conflict points, and paths. }
\label{fig:Intersection}
\vspace{-12pt}
\end{figure}

\section{ Problem Formulation}

We consider a signal-free intersection illustrated in Fig. \ref{fig:Intersection}, where a coordinator handles communication with the CAVs within a given range. We call this communication region \textit{control zone}. In our formulation, we consider that CAVs do not perform any lane-change maneuver, thus there exist a finite number of paths within the control zone. These paths constitute the set $\mathcal{L}=\{1,2,\ldots, z\},\; z \in \mathbb{N}$. The points where the roads intersect and lateral collisions may occur are called conflict points and belong to the set $\mathcal{O}_z \subseteq \mathbb{N}$ for each $z\in\mathcal{L}$. The set $\mathcal{N}(t) = \{1, 2, \ldots, N(t)\}$ denotes the queue of vehicles inside the control zone at time $t$, where $N(t)$ denotes the total number of vehicles within the control zone at time $t$.

We consider that the dynamics of each CAV $i \in \mathcal{N}(t)$ are described by a double integrator model, i.e.,
\begin{equation}\label{eq:model2}
\begin{split}
\dot{p}_{i}(t) &= v_{i}(t), 
\\
\dot{v}_{i}(t) &= u_{i}(t), 
\end{split} 
\end{equation}
where ${p_{i}\in\mathcal{P}}$, ${v_{i}\in\mathcal{V}}$, and
${u_{i}\in\mathcal{U}}$ denote the longitudinal position of the rear bumper, speed, and control input (acceleration) of the vehicle, respectively. 
The sets ${\mathcal{P}, \mathcal{V},}$ and $\mathcal{U}$ are compact subsets of $\mathcal{R}$. 
The control input is bounded by 
\begin{equation}\label{eq:uconstraint}
u_{\min} \leq u_{i}(t)\leq u_{\max},\quad \forall i \in \mathcal{N}(t),
\end{equation}
where ${u_{\min}<0}$ and ${u_{\max}>0}$ are the minimum and maximum control inputs, respectively, as designated by the physical acceleration and braking limits of the vehicles, or the limits of passenger comfort, whichever is more restrictive.
Next, we consider the speed limits of the CAVs, 
\begin{equation}\label{eq:vconstraint} 
v_{\min} \leq v_i(t) \leq v_{\max},\quad \forall i \in {\mathcal{N}(t) }, 
\end{equation}
where ${v_{\max}>v_{\min}>0}$ are the minimum and maximum allowable speeds. This implies that a CAV is not allowed to stop within the intersection.

\subsection{Safety Constraints}
Before we delve into the safety constraints, we define some useful notation for each path $z\in\mathcal{L}$.
To simplify our notation, we omit the subscript $z$ on variables where it does not lead to ambiguity.
Let \( p_i^0 \) be the entry point in the intersection of each CAV \( i \), \( p_i^n \) be the position at the conflict point \( n \in \mathcal{O}_z = \{1,2,3\} \) (note that each path $z$ in Fig. \ref{fig:Intersection} passes through 3 conflict points), and \( p_i^f \) be the exit point of the intersection. 
Let \( t_i^0 \in \mathbb{R}_{\ge0} \) be the time at which CAV \( i \) enters the control zone and \( t_i^{\text{n}} \in \mathbb{R}_{\ge0} \) be the time at which it reaches the conflict point \( n \in \{1,2,3\} \); and let \( t_i^{\text{f}} \in \mathbb{R}_{\ge0} \) be the time when it exits the control zone. 

Since the speed \( v_i(t) \) is bounded below by a positive number $v_{\text{min}}>0$, the position \( p_i(t) \) monotonically increases \cite{Malikopoulos2020}. Then, the time trajectory for CAV $i$ can be defined as \( t_i = p_i^{-1} \), which allows us to determine \( t_i^n \), the time at which CAV $i$ arrives at conflict point \( n \).

To avoid conflicts between vehicles in the control zone, we need to impose the following safety constraints: (1) the rear-end constraints between vehicles on the same path, and (2) the lateral safety constraints between vehicles traveling on different intersecting paths.
To ensure rear-end collision avoidance between a CAV $i \in \mathcal{N}(t)$ and its leading CAV $k \in \mathcal{N}(t) \setminus \{i\}$, we impose that
\begin{equation}\label{rearend}
t_k(p) - t_i(p) \geq \tau_{r},
\end{equation}
where $\tau_{r}$ denotes a safe time headway between two vehicles.
For lateral collision avoidance, we consider a scenario where CAV $k \in \mathcal{N}(t)\setminus\{i\}$ has a planned trajectory potentially leading to a conflict with CAV $i$. In such a scenario, we need to investigate the two following cases:

Case 1. CAV $i$ reaches the conflict point $n$ after CAV $k$. Then, we require
\begin{equation}\label{lateral1}
t^n_i - t_i(p) \geq \tau_{\ell} \quad \forall p \in [p^0_i, p^n_k],
\end{equation}
where $t_i^n$ denotes the time CAV $k$ arrives at conflict point $n$. This constraint ensures that the time headway between CAV \(i\) and conflict point \(n\) is greater than or equal to $\tau_{\ell}$ until CAV \(k\) has successfully passed through conflict point $n$.

Case 2. CAV $i$ arrives at conflict point $n$ before CAV $k$. Then, we require
\begin{equation}\label{lateral2}
t^n_k - t_k(p) \geq \tau_{\ell} \quad \forall p \in [p^0_k, p^n_i].
\end{equation}




Next, we review the optimal trajectory of a CAV $i$ by using the two-level optimization framework presented in \cite{Malikopoulos2020}. The upper-level framework focuses on determining, for every CAV \(i\in\mathcal{N}(t)\), the shortest possible time \(t_i^f\) to exit the control zone, given its desired destination. The low-level framework involves an optimization problem, the solution of which derives the optimal control input for CAV $i\in\mathcal{N}(t)$, given $t_i^f$, while adhering to constraints related to the vehicle's state, control inputs, and safety.

\subsection{Low-Level Optimization}

We start our exposition by reviewing the framework presented in \cite{Malikopoulos2020} that can aim at deriving the energy-optimal trajectory for each CAV $i$.
We consider that for every CAV $i \in {\mathcal{N}}(t),$ the time $t_i^f$ (exiting time of the control zone) is known. Then, the energy-optimal control problem finds the optimal input by solving the following optimization problem.
\begin{problem}
(\textbf{Energy-optimal control problem})  \label{prb:ocp-1}
\begin{equation}
\label{eq:energy_cost}
\begin{split}
&\underset{u_i\in\mathcal{U}}{\min} \quad \frac{1}{2} \int_{t^{0}_{i}}^{t_i^f} u^2_i(t) \, \mathrm{d}t, \\
&\text{subject to:} \\
& \quad \eqref{eq:model2}, \eqref{eq:uconstraint}, \eqref{eq:vconstraint}, \eqref{rearend},\eqref{lateral1},\eqref{lateral2},\\
& \text{given:} \\
& \quad p_i (t_i^0) = p^0, \,\, v_i (t_i^0) = v_i^0, 
 \,\, p_i (t_i^f) = p^f,
\end{split}
\end{equation}

\noindent where $v_i^0$ is the speed of CAV $i$ at $t_i^0$.
The boundary conditions in \eqref{eq:energy_cost} are set at the entry and exit of the control zone. 
\end{problem}

The closed-form solution of this problem for each CAV $i$ can be derived using the Hamiltonian analysis as presented in \cite{Malikopoulos2020}. The optimal unconstrained trajectory is as follows:
\begin{equation}\label{eq:optimalTrajectory}
\begin{split}
u_i(t) &= 6 \phi_{i,3} t + 2 \phi_{i,2}, \\
v_i(t) &= 3 \phi_{i,3} t^2 + 2 \phi_{i,2} t + \phi_{i,1}, 
\\
p_i(t) &= \phi_{i,3} t^3 + \phi_{i,2} t^2 + \phi_{i,1} t + \phi_{i,0},
\end{split}
\end{equation} 
where $\phi_{i,3}\in\mathbb{R}_{\ge0}$ and $\phi_{i,2}, \phi_{i,1}, \phi_{i,0} \in \mathbb{R}$ are constants of integration. Given the boundary conditions in Problem \ref{prb:ocp-1} 
, the optimal boundary condition ${u_i(t_i^f) = 0}$ and considering $t_i^{f}$ is known, the constants of integration can be found by:
\begin{equation} \label{matrix1}
\bbsym{\phi}_i =
\begin{bmatrix}
\phi_{i,3}
\\
\phi_{i,2}
\\
\phi_{i,1}
\\
\phi_{i,0}
\end{bmatrix}
= 
\begin{bmatrix}
(t_i^0)^3 & (t_i^0)^2 & t_i^0 & 1 
\\
3(t_i^0)^2 & 2t_i^0 & 1 & 0 
\\
(t_i^f)^3 & (t_i^f)^2 & t_i^f & 1 
\\
6t_i^f & 2 & 0 & 0 
\end{bmatrix}^{-1}
\begin{bmatrix}
p^0 
\\
v_i^0 
\\
p^f 
\\
0
\end{bmatrix}.
\end{equation}
Having defined Problem 1 we are ready to present the upper-level optimization problem where we compute $t_i^f$ which is then passed as an input in Problem 1.

\vspace{-5pt}
\subsection{Upper-level optimization problem}
\begin{problem}
(\textbf{Time-optimal control problem}) 
At the time $t_i^0$ of entering the control zone, let ${\mathcal{F}_i(t_i^0)=[\underline{t}_i^f, \overline{t}_i^f]}$ be the feasible range of travel time under the state and input constraints of CAV $i$ computed at $t_i^0$. The formulation for computing $\underline{t}_i^f$ and $\overline{t}_i^f$ can be found by letting the CAV accelerate or decelerate at the maximum rate, respectively. Then CAV $i$ solves the following time-optimal control problem to find the minimum exit time $t_i^f \in \mathcal{F}_i(t_i^0)$ that satisfies all state, input, and safety constraints:
\begin{align*}
&\underset{t_i^f\in\mathcal{F}_i(t_i^0)}{\min} \quad  t_i^f \\
&\text{subject to:} \\
& \quad \eqref{eq:model2}, \eqref{eq:uconstraint}, \eqref{eq:vconstraint}, \eqref{rearend},\eqref{lateral1}, \eqref{lateral2}\\
&\text{given:} \\& \quad p_i(t_i^0) = p_0, v_i(t_i^0) = v_i^0, p_i(t_i^f) = p_f, u_i(t_i^f) = 0
\end{align*}
\end{problem}

The computation steps for numerically solving Problem 2 are summarized as follows. First, we initialize $t_i^f = \underline{t}_i^f$, and compute the parameters $\phi_i$ using \eqref{matrix1}. We evaluate all the state, control, and safety constraints. If none of the constraints is violated, we return the solution; otherwise, $t_i^f$ is increased by a step size. The procedure is repeated until the solution satisfies all the constraints. To visualize the process, observe Fig. \ref{fig:Hamiltonian}, which illustrates a snapshot where a CAV $i$ joins the intersection. The dotted curves show the trajectories CAV $i$ can take for $t_i^f$ in the time interval $[\underline{t}_i^f,\overline{t}_i^f]$. Supplemental video describing the process can be found on the \href{https://sites.google.com/cornell.edu/feasibility/home}{\underline{paper's website}}.

Finally, by solving Problem 2, the optimal exit time $t_i^f$ along with the optimal trajectory and control law (12) are obtained for CAV $i$ for $t \in [t_i^0, t_i^f]$.

\begin{figure}[t!]
\centering
\includegraphics[width=0.4\textwidth]{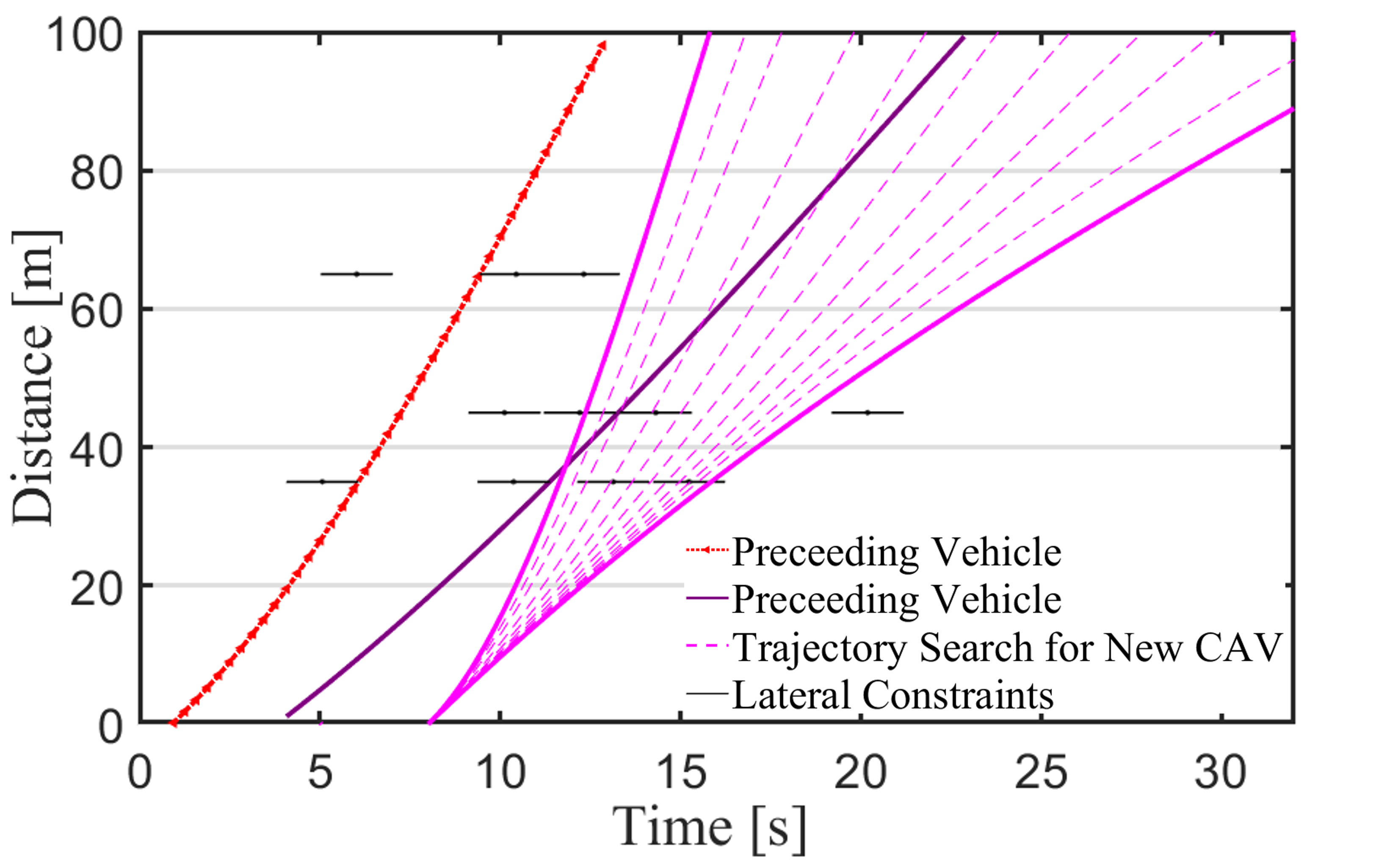}
\caption{Snapshot of a CAV entering the Intersection }
\label{fig:Hamiltonian}
\vspace{-16pt}
\end{figure}

\begin{remark} If a feasible solution to Problem 2 exists, then the solution is a cubic polynomial that guarantees none of the constraints become active. In case the solution to Problem 2 does not exist due to congestion, we derive the optimal trajectory for the CAVs by piecing together the constrained and unconstrained arcs until the solution does not violate any constraints. However, piecing together arcs is rather computationally expensive. So, in Section III, we circumvent this issue using the concept of numerical interpolation.
\end{remark}

\vspace{-7pt}

\section{Enhancing the Domain of Feasible Solutions}
In this section, we present an approach that enhances the CAV trajectories' feasible domain resulting from the solution of Problem 1. The trajectories provided by Problem 1 are energy-optimal but constrained to be $3^{rd}$ order polynomials.

\vspace{-7pt}

\subsection{Theoretical Results}
In the subsection, we provide some theoretical results needed for our exposition.


\begin{theorem}\label{th_interpolation}
Given \( n+1 \) distinct nodes, \(\{x_0, \ldots, x_{n}\}\) and \( n+1 \) corresponding values \(\{y_0, \ldots, y_{n}\}\), there exists a unique polynomial \( f(x) \) of degree \( n \), such that \( f(x_i) = y_i \) for all \( i = 1, \ldots, n+1 \).
\end{theorem}
\begin{proof}
    See \cite{quarteroni2006numerical}, page 334, Theorem 8.1.
\end{proof}

\begin{theorem}
    The coefficients of a $n^{th}$ order polynomial given $n$ points are given by the following equation where the square matrix is a Vandermonde matrix:
\end{theorem}

\begin{equation}
\begin{bmatrix}
1 & x_0 & x_0^2 & \cdots & x_0^{n-1} \\
1 & x_1 & x_1^2 & \cdots & x_1^{n-1} \\
\vdots & \vdots & \vdots & \ddots & \vdots \\
1 & x_{n-1} & x_{n-1}^2 & \cdots & x_{n-1}^{n-1} \\
\end{bmatrix}
\begin{bmatrix}
\phi_{0} \\
\phi_{1} \\
\vdots \\
\phi_{n-1} \\
\end{bmatrix}
=
\begin{bmatrix}
y_0 \\
y_1 \\
\vdots \\
y_{n-1} \\
\end{bmatrix}
\label{Vandermonde}
\end{equation}

\begin{proof}
See \cite{shen2022polynomial}, Chapter 2, Equation (5).
\end{proof}

To exploit the results provided by Theorems 1 and 2, we need to ensure that the Vandermonde matrix is invertible, which is not guaranteed in general since it depends on the values \(x_0, \ldots, x_{n-1}\). The following result shows that a Vandermonde matrix has a specific determinant form. This specific form will assist us in proving that the Vandermonde matrix, as defined in \text{\eqref{Vandermonde}}, is invertible.

\begin{lemma}
    The determinant of a $n \times n$ Vanderomde matrix is equal to $\det(V) = \prod_{0 \leq i < j \leq n-1} (x_j - x_i)$.
\end{lemma}
\begin{proof}
We use mathematical induction to prove this result. For $n=2$, we have $\det(V) = (x_1 - x_0)$. Assume the formula holds for a $(n-1) \times (n-1)$ matrix. For $n \times n$, perform row operations by subtracting from each row $i > 1$ the first row scaled by $x_{i-1}$, thus leaving a leading 1 and zeros below in the first column and extracting $(x_{i-1} - x_0)$, which does not alter $\det(V)$. This process yields an $(n-1) \times (n-1)$ Vandermonde matrix, for which, by induction, the determinant is the product of $(x_j - x_i)$ for $1 \leq i < j \leq n-1$. Multiplying by the initially factored terms completes the induction. 
\end{proof}

Next, we prove that the determinant provided in Lemma 1 is not equal to zero. Given the nature of our application, we consider that the values \(x_0, \ldots, x_{n-1}\) represent times along a CAV's trajectory, and \(y_0, \ldots, y_{n-1}\) represent the corresponding positions. Then, having known that the time trajectory is strictly increasing ($v_{min}>0$), it is sufficient to show that the matrix is invertible if $0<x_0<x_1<\ldots<x_{n-1}$. Let us prove that in the following theorem.

\begin{theorem}
The Vandermonde matrix presented in \eqref{Vandermonde} is an invertible matrix if $0<x_0<x_1<\ldots<x_{n-1}$.    
\end{theorem}

\begin{proof}
    From Lemma 1, the determinant of a Vandermonde matrix $V$ is $\det(V) = \prod_{0 \leq i < j \leq n-1} (x_j - x_i).$ Since $x_0<x_1<\ldots<x_{n-1}$, it follows that for all $i, j$ with $0 \leq i < j \leq n-1$, we have $x_j - x_i \neq 0$. Therefore, $\prod_{0 \leq i < j \leq n-1} (x_j - x_i)\neq 0.$ A non-zero determinant implies that the matrix $V$ is invertible and the proof is complete.
\end{proof}

\vspace{-8pt}
\subsection{Implementation in CAVs}
Next, we apply the results above to improve the feasibility of the intersection. Consider a CAV $i$ that cannot find a feasible trajectory after solving Problem 2. Then we aim to construct a $n_{th}$ order polynomial using Theorems 1, 2, and 3, which corresponds to a feasible position trajectory \( p(t) \) of CAV \( i \), as defined in  \eqref{eq:model2}. In other words, our objective is to identify $n$ points $(p_i, t_i)$ along the trajectory of CAV $i$ to interpolate. These nodes enable the use of \eqref{Vandermonde} and must ensure that the resulting polynomial adheres to all constraints and is energy efficient. 

We choose to construct a 4\textsuperscript{th} order polynomial, inspired by the fact that we already know the location of the position nodes \(p_i^0\), \(p_i^1\), \(p_i^2\), \(p_i^3\), and \(p_i^f\). Given that we know these position nodes, we are searching for the corresponding times  \(t_i^1\), \(t_i^2\), \(t_i^3\), and \(t_i^f\) such that we can compute the coefficients $\phi_i$ of the $4^{th}$ order polynomial using \eqref{Vandermonde}. Recall that \(t_i^0\) is considered known, as it represents the time of entry into the control zone. 

\begin{remark}
Note that our approach generalizes the $3^{\text{rd}}$ order polynomial solution and can yield solutions in a broader case of traffic scenarios.
Furthermore, for systems with $n^{th}$ order integrator dynamics, the control-minimizing trajectory is a $2(n-1)$ order polynomial \cite{beaver2023lq}.
Thus, we impose a smoother trajectory that minimizes energy consumption, effectively handling the trade-off between passenger comfort and energy efficiency.
\end{remark}

Next, we discuss how our approach can efficiently address lateral and rear-end constraints.

\vspace{-6pt}
\subsection{Lateral constraints}
In Section II.A, we pointed out that the lateral constraints for every CAV $i$ are only associated with the points \((p_i^1, t_i^1), (p_i^2, t_i^2),(p_i^3, t_i^3)\). 
However, to interpolate from appropriate times \( t_i^1, t_i^2, \) and \( t_i^3 \) for a CAV \( i \) on path $z\in\mathcal{L}$, it is essential to ensure that there is always a non-empty set in which \( t_i^1, t_i^2, \) and \( t_i^3 \) can belong to. Namely, we want to show that for any path $z\in\mathcal{L}$, there is a sufficiently large gap between vehicles at each node. Next, we provide a condition for the existence of such a gap.

\begin{proposition} \label{proposition}
Consider a CAV \(i \in \mathcal{N}(t)\) and a preceding CAV \(k \in \mathcal{N}(t) \setminus \{i\}\) on the same path $z\in\mathcal{L}$. If we enforce $\tau_{r}>=2\tau_{\ell}$, where $\tau_{r}$ and $\tau_{\ell}$ are defined in \eqref{rearend}, \eqref{lateral1} and \eqref{lateral2},  then  for every conflict point $n\in\{1,2,3\}$ along path $z$, there exists a non-empty feasible time set \(\mathcal{M}_n = [\underline{\tau}, \overline{\tau}] \neq \emptyset\), within which a CAV \(j\) from a path conflicting with path $z$ can safely pass between the two CAVs \(i\) and \(k\).
\end{proposition}

\begin{proof}
From \eqref{rearend}, the minimum headway between CAVs $i$ and $k$ is $\tau_{r}$. Given $\tau_{r}\geq2\tau_{\ell}$ where $\tau_{\ell}$ is the safety headway for lateral conflict \eqref{lateral1} and \eqref{lateral2}, there is always feasible gap between CAVs $i$ and $k$ such that a CAV $j$ from a conflicting path can safely pass between $i$ and $k$. 
\end{proof} 

To visualize this result, we note in Fig. \ref{fig:Relatioship} that between CAV $i$ and $k$ there is always time for a CAV $j$ from a conflicting path to pass safely.


\begin{figure}[t!]
\centering
\includegraphics[width=0.33\textwidth]{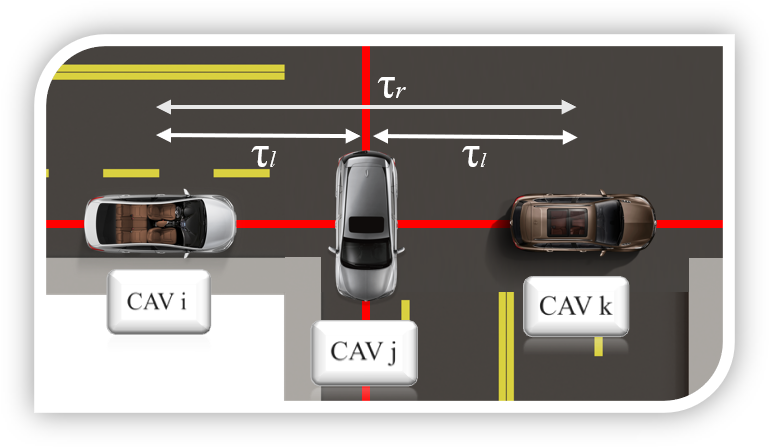}
\caption{Relationship between minimum rear-end and lateral time headways}
\label{fig:Relatioship}
\vspace{-12pt}
\end{figure}

\begin{remark}
We have ensured the existence of \(t_i^1, t_i^2,\) and \(t_i^3\) for the lateral safety constraint. Given that the lateral constraints can only occur at points \((p_i^1, t_i^1), (p_i^2, t_i^2),(p_i^3, t_i^3)\), and given that the function $p(t)$ is strictly increasing, the final trajectory cannot make any of the lateral constraints become active.
\end{remark}

\begin{remark}
Proposition 1 enhances the feasibility of a CAV to navigate between two other CAVs from a conflicting path. While the inequality \(\tau_{r} \geq 2\tau_{\ell}\) mandates that vehicles on the same path maintain a greater distance, this requirement, not included in \cite{Malikopoulos2020}, actually broadens the practical applicability of our approach. This is because it ensures we can invariably identify time points \(t_i^n\) that adhere to the lateral constraints.

\end{remark}

\vspace{-6pt}
\subsection{Rear-End Constraints}
Although we have proven that \(t_i^1\), \(t_i^2\), and \(t_i^3\) can always be feasible to guarantee lateral constraints, we must also define conditions for the rear-end constraints. An analytical approach offering such a guarantee is non-trivial, given the polynomial's design through \eqref{Vandermonde} and has not been explicitly addressed in the literature to date. Although the resulting polynomial does not make the rear-end constraints active, when for a CAV \(i\) and a preceding CAV \(k\) the condition \(t^n_i \geq t^n_k + \tau_{r}\) is met, we lack theoretical guarantees for such a result. In instances where rear-end constraints may become active, the approach presented in \cite{bang2023optimal} (see Section 4.2) can be employed. 
Finally, our solution must also adhere to vehicle speed and acceleration constraints, including their maximum and minimum limits. These constraints are discussed in the following subsection.

\vspace{-6pt}
\subsection{Solution}

In this subsection, we analyze the solution approach. Recall that we aim to minimize the jerk of the position trajectory. Therefore, let us define a function \( g : \mathbb{R}^5 \to \mathbb{R}\) which takes as parameters the vector \( \{t_i^0, t_i^1, t_i^2, t_i^3, t_i^f \} \), computes the 4th order polynomial trajectory \( p_i(t) \) using \eqref{Vandermonde} and returns the jerk of that polynomial, which is the third derivative of \( p_i(t) \). So, we aim to find the optimal values of $t_i^1, t_i^2, t_i^3, t_i^f$ to minimize the jerk of CAV $i$. Recall that $t_i^0$ is known.

Before introducing the optimization framework that addresses this problem, we denote the feasible time set for each parameter $t_i^n$ as $\mathcal{T}_i^n,\;n\in\{1,2,3,f\}$. Note that each time set $\mathcal{T}_i^n,n\in\{1,2,3\}$ consists of a union of disjoint compact time sets due to the existence of lateral constraints associated with the set of conflict points $\mathcal{O}_z=\{1,2,3\}$. In other words, for each conflict point $n\in\mathcal{O}_z$ we have  $\mathcal{T}_{i}^n=\{\mathcal{T}_{i,1}^n,\mathcal{T}_{i,2}^n,...,\mathcal{T}_{i,\mathcal{D}}^n\}$. For example, in Fig. \ref{fig:Hamiltonian}, the red arrowed curve refers to a position trajectory of CAV $i$ that passes through the disjoint sets $\mathcal{T}_{i,1}^1$, $\mathcal{T}_{i,0}^2$ and $\mathcal{T}_{i,1}^3$. 


Conversely, the exit time $t_i^f$ is subject to rear-end constraints or the lower bound $\underline{t}_i^f$ whichever is more restrictive, and the upper bound $\overline{t}_i^f$. This ensures that the set of feasible exit times, denoted as $\mathcal{T}_i^f$, forms already a compact set.

Considering all the disjoint sets in $\mathcal{T}_{i}^n$ for each conflict point $n\in\{1,2,3\}$ as the feasibility domain of $t_i^n$ does not facilitate the optimization process. 
Consequently, our goal is to limit the feasibility set for each $t_i^n$ to a single convex set. To achieve this, we need to select only one set among the multiple disjoint sets included in $\mathcal{T}_i^n$.
For each CAV $i$ and for each conflict point $n\in \mathcal{O}_z$, we introduce Algorithm 1 to achieve the selection of a single convex set ${\mathcal{T}_i^{n}}^* \subseteq \mathcal{T}_i^n$, focusing on jerk minimization.

\begin{algorithm}
\caption{Selection of ${\mathcal{T}_i^{n}}^*$ when CAV $i$ joins the intersection and Problem 2 does not offer a solution.}
\label{alg:optimal_trajectory_selection}
\begin{algorithmic}
\ENSURE Selection of an efficient ${\mathcal{T}_i^{n}}^*$ for all $n \in \{1,2,3\}$.
\FOR{each $n\in\{1,2,3\}$}
    \FOR{each disjoint set $\mathcal{T}_{i,j}^n \in \mathcal{T}_i^n$}
        \STATE Average time: $t_{i,\text{avg},j}^n = \frac{\min(\mathcal{T}_{i,j}^n) + \max(\mathcal{T}_{i,j}^n)}{2}$.
    \ENDFOR
\ENDFOR
\STATE 1. Evaluate all different $p_i(t)$ based on $t_{i,\text{avg},j}^n \forall n \in \{1,2,3\}$.
\vspace{-10pt}
\STATE 2. Identify the combination that has the minimum jerk.
\STATE 3. Determine the optimal convex set ${\mathcal{T}_i^{n}}^*\;\forall n \in \{1,2,3\}$, based on the identified combination in step 2.
\end{algorithmic}
\end{algorithm}

Now that we have defined a convex feasible set for each $t_i^n$ we are in position to present our minimization problem.


\begin{problem}{\textbf{Energy-optimal Control Problem}}
\begin{align}
&\min_{t_i^1, t_i^2, t_i^3, t_i^f } \quad  \int_{t_i^0}^{t_f} (g(t;t_i^0, t_i^1, t_i^2, t_i^3, t_i^f))^2 \,dt \\
&\text{subject to:} \nonumber\\
&t_i^0\; \text{given},\; t_i^1 \in {\mathcal{T}_i^1}^*,\; t_i^2 \in {\mathcal{T}_i^2}^*, \;t_i^3 \in {\mathcal{T}_i^3}^*,\; t_i^f \in {\mathcal{T}_i^f}^*, \nonumber\\
&  p_i(t_i^0) = p_i^0,\; v_i(t_i^0) = v_i^0, \; \eqref{eq:model2}, \; \eqref{eq:uconstraint},\; \eqref{eq:vconstraint}.
\end{align}
\end{problem}


Given that the feasibility domain of Problem 3 forms a convex set, it streamlines the search process, facilitating a rapid acquisition of a numerical solution. The source code for this optimization problem on the \href{https://sites.google.com/cornell.edu/feasibility/home}{\underline{paper's website}}. 

Our analysis has focused on defining a position trajectory for cases where the solution provided by Problem 2 is not viable. It is important to note that while it is always possible to construct such a polynomial that satisfies read-end and lateral constraints, in rare instances, there is no combination of $t_i^n$ that can lead to a solution that satisfies the constraints \eqref{eq:uconstraint} and \eqref{eq:vconstraint}. In other words, although our approach expands the feasibility domain provided in \cite{Malikopoulos2020}, there might be cases where a feasible solution is still not viable due to an extremely highly congested environment. There are several strategies to address these scenarios. One method involves implementing dynamically changing speed limits based on intersection traffic conditions. Another strategy is establishing initial state conditions ensuring compliance with all constraints. Then, we can define a buffer zone before the control zone in which the vehicles adjust their state. Also, integrating these constraints directly into the polynomial coefficients is a viable approach. Current research is directed towards exploring these solutions.


\section{Numerical Simulations}

To validate our framework, we conducted numerical simulations using Matlab and PTV VISSIM software. We modeled an intersection, shown in Fig.~\ref{fig:Intersection}, with each path \(z \in \mathcal{L}\) being $100$ m in length. In Fig.~\ref{fig:Results only Hamiltonian}, we demonstrate for path 6, the feasibility benefits offered by Proposition 1, which are achieved without the need for polynomial interpolation for any of the four trajectories. Notably, the zoomed-in frames reveal that there is always a sufficient gap between lateral constraints at each conflict point \(n\), enabling an incoming CAV to define its trajectory—a consideration not accounted for in \cite{Malikopoulos2020}. Conversely, Fig.~\ref{fig:Results Interpolation} presents a scenario on path 3, where unconstrained trajectories by Problem 2, proved to be unfeasible for 2 CAVs. As a result, we resorted to addressing Problem 3 by interpolating a polynomial aimed at minimizing jerk. Upon examining the trajectories of the CAVs with the dotted curves, it becomes evident that such trajectories could not have been achieved using a \(3^{rd}\) order polynomial. The execution time for defining such polynomials using Problem 3 was \(9 \times 10^{-3}\) seconds with an Intel i9 processor with 64 GB RAM and a 3.4 GHz clock speed. We encourage interested readers to visit the \href{https://sites.google.com/cornell.edu/feasibility/home}{\underline{paper's website}} to find supplemental simulation results provided by PTV VISSIM.

\begin{figure}[t!]
\centering
\includegraphics[width=0.5\textwidth]{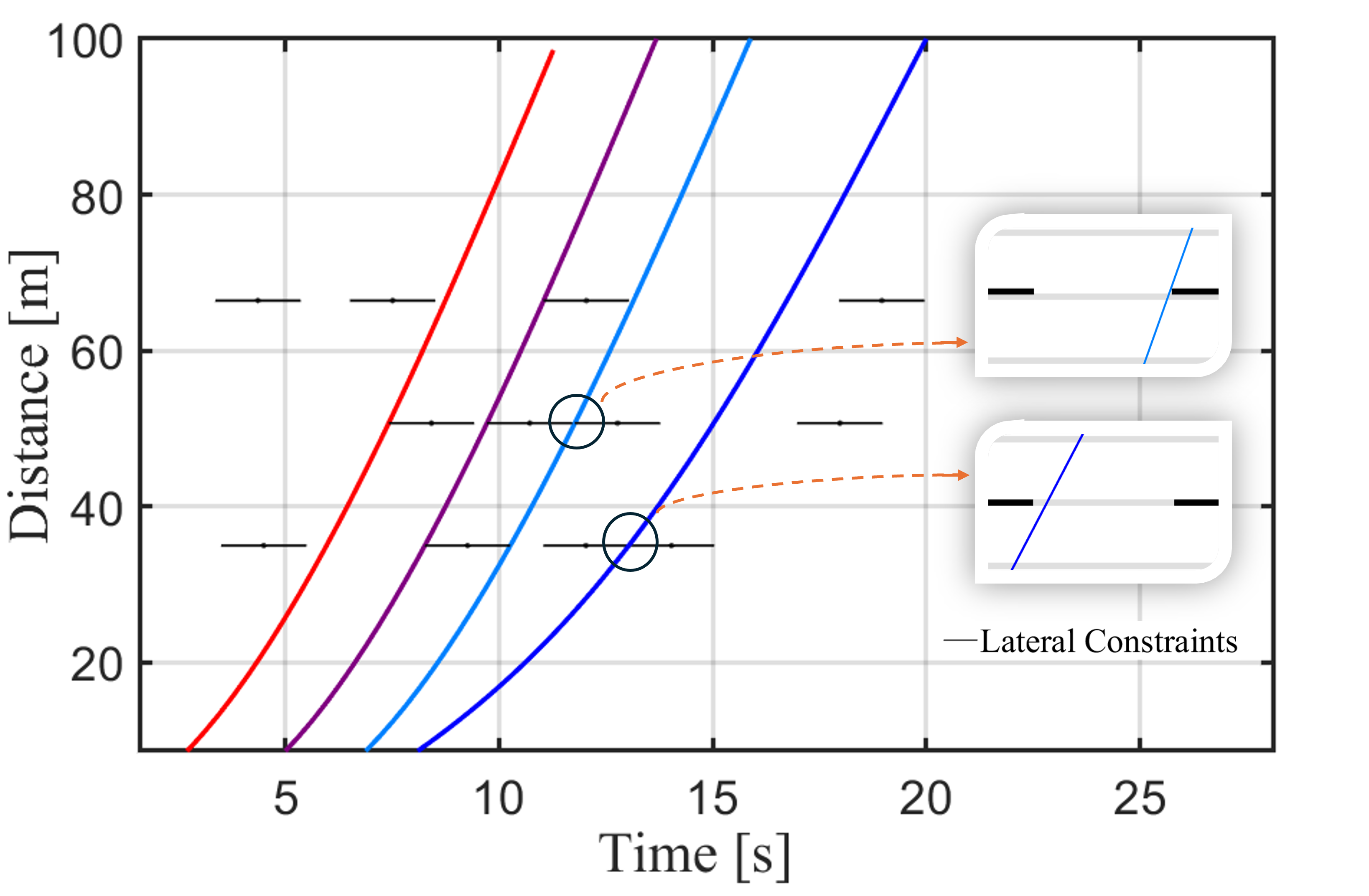}
\caption{CAV trajectories defined only from Problem 2.}
\label{fig:Results only Hamiltonian}
\vspace{-7pt}
\end{figure}

\begin{figure}[t!]
\centering
\includegraphics[width=0.5\textwidth]{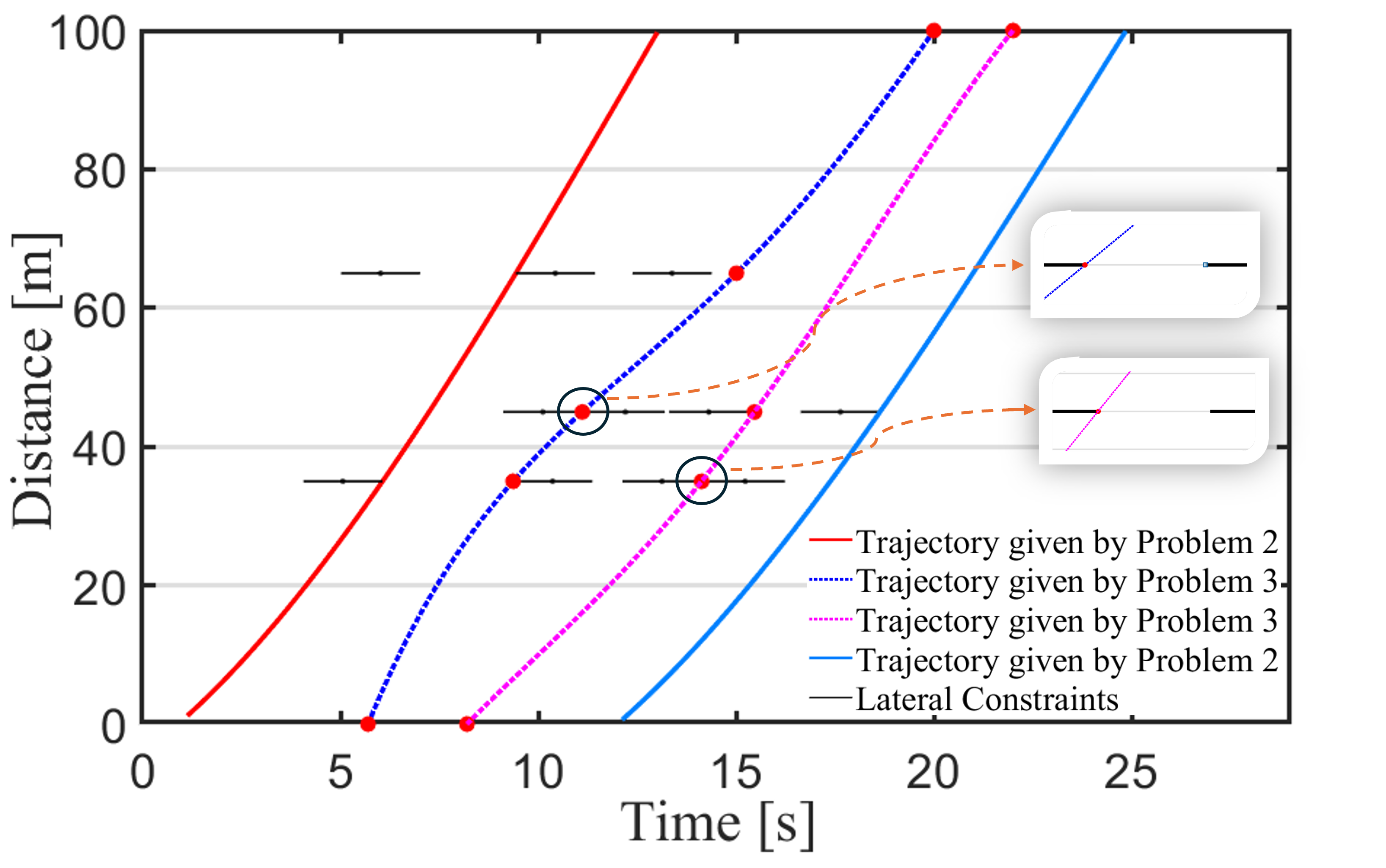}
\caption{CAV trajectories defined from both Problem 2 and 3.}
\label{fig:Results Interpolation}
\vspace{-15pt}
\end{figure}

\vspace{-10pt}

\section{Concluding Remarks}
In this letter, we addressed the challenge of expanding the feasible solution domain within a decentralized control framework for CAVs at signal-free intersections. The proposed approach employs numerical interpolation techniques to establish a feasible trajectory for CAVs crossing a signal-free intersection. Our findings demonstrate that our approach can extend the feasibility domain and provide a real-time solution. A potential direction for future research is generalizing this method for more complex traffic scenarios and integrating it within mixed traffic environments.

\linespread{0.99}\selectfont
\bibliographystyle{ieeetr}
 \bibliography{bibliography.bib,IDS_Publications_01012024}

\end{document}